\documentclass[12pt]{amsart}

\newtheorem{theorem}{Theorem}[section]
\newtheorem{corollary}[theorem]{Corollary}

\theoremstyle{definition}

\theoremstyle{remark}

\numberwithin{equation}{section}

\newcommand{\p}{\partial}
\newcommand{\n}{\nabla}
\newcommand{\z}{\bar{z}}
\newcommand{\der}[1]{\frac{\partial}{\partial #1}}
\newcommand{\fder}[2]{\frac{\partial #1}{\partial #2}}

\begin{document}
\date{}
\title[Uniqueness theorems for free boundary minimal disks in space forms]
{Uniqueness theorems for free boundary minimal disks in space forms}
\author{Ailana Fraser}
\address{Department of Mathematics \\
                 University of British Columbia \\
                 Vancouver, BC V6T 1Z2}
\email{afraser@math.ubc.ca}
\author{Richard Schoen}
\address{Department of Mathematics \\
                 University of California \\
                 Irvine, CA 92617}
\email{rschoen@math.uci.edu}
\thanks{2010 {\em Mathematics Subject Classification.} 53A10. \\
A. Fraser was partially supported by  the 
Natural Sciences and Engineering Research Council of Canada and R. Schoen
was partially supported by NSF grant DMS-1404966.}

\begin{abstract} 
We show that a minimal disk satisfying the free boundary condition in a constant curvature 
ball of any dimension is totally geodesic. We weaken the condition to parallel mean
curvature vector in which case we show that the disk lies in a three dimensional 
constant curvature submanifold and is totally umbilic. These results extend to
higher dimensions earlier three dimensional work of J. C. C. Nitsche and R. Souam.

\end{abstract}

\maketitle

\section{Introduction}
In this short note we consider the free boundary minimal disks
in a ball in Euclidean space or a space $N^n$ of constant curvature. These are proper
branched minimal immersions of a disk into the ball that meet the boundary
orthogonally (the conormal vector of the disk is normal to the boundary of the
ball). Such surfaces have been extensively studied and they arise as extremals
of the area functional for relative cycles in the ball. They also arise as extremals
of a certain eigenvalue problem \cite{FS1} and \cite{FS2}. We show here that
any such free boundary minimal disk is a totally geodesic disk passing through the 
center of the ball. 

We extend the result to free boundary disks with parallel mean curvature in a ball. 
If the mean curvature is not zero, we show that such disks are contained in a totally
geodesic three dimensional submanifold (an affine subspace in the flat case)
and that the disk is totally umbilic.

Both of these results are known in case $n=3$. For the euclidean case, both the minimal
and constant mean curvature results were obtained by J. C. C. Nitsche \cite{Ni}. In the case of space
forms, the results were obtained for $n=3$ by R. Souam \cite{S}, and were generalized
to constant contact angle boundary conditions by A. Ros and R. Souam \cite{RS}. 

The method of proof for the results is complex analytic as in the case $n=3$. The
general idea goes back to H. Hopf and involves the construction of holomorphic
differentials which then vanish under suitable conditions on the surface and suitable
boundary conditions. In the $n=3$ case, the Hopf differential is a holomorphic quadratic
differential and the free boundary condition implies that it is real along the boundary
in the sense that it has real values when applied to the unit tangent vector. Geometrically,
this follows from the condition that the boundary be a principal curve. In higher
dimensions the corresponding differential has values in the normal bundle, but there
is an associated complex valued quartic differential gotten by squaring the Hopf differential
using the inner product. Our result is obtained by showing vanishing of this quartic
differential and discussing the consequences.

There is an analog between minimal submanifolds of spheres and free boundary
minimal submanifolds of the ball. Under this analog the Nitsche theorem is analogous
to a result of Almgren \cite{A1} which says that a minimal $\mathbb S^2$ in $\mathbb S^3$ is
totally geodesic. It is well known (see Calabi \cite{C}) that there are many minimal immersions of $\mathbb S^2$
in $\mathbb S^n$ for $n\geq 4$ which are not totally geodesic. From the point of view
of this analogy our results are surprising. To make a more definite connection 
between the two problems, we point out that our results imply that any minimal $\mathbb S^2$
in $\mathbb S^n$ which is invariant under reflection through a totally geodesic 
$\mathbb S^{n-1}$ must be totally geodesic.

\section{Uniqueness results}

\begin{theorem}
Let $u: D \rightarrow B^n$ be a proper branched minimal immersion, such that $u(D)$ meets $\p B$ orthogonally. Then $u(D)$ is an equatorial plane disk.
\end{theorem}

\begin{proof}
Since $u: D \rightarrow B^n$ is a minimal immersion, $u$ is harmonic, $u_{z\bar{z}}=0$, and conformal, $u_z \cdot u_z=0$. Since $u$ is harmonic, it follows that
\[
      (u_{zz} \cdot u_{zz} )_{\bar{z}} =0.
\]
Note that the normal bundle is smooth across branch points, and
\begin{equation}  \label{decomp}
       u_{zz}=u_{zz}^\perp + \frac{u_{zz} \cdot u_z}{|u_{\bar{z}}|^2} u_{\bar{z}}
                                             + \frac{u_{zz} \cdot u_{\bar{z}}}{|u_{z}|^2} u_{z}
\end{equation}
where $u_{zz}^\perp$ denotes the component of $u_{zz}$ orthogonal to $\Sigma=u(D)$. But
\[
       u_{zz} \cdot u_z=\frac{1}{2} (u_z \cdot u_z)_z=0
\]
since $u$ is conformal, and thus by (\ref{decomp})
\[
      u_{zz}^\perp \cdot u_{zz}^\perp =u_{zz}\cdot u_{zz}.
\]
In polar coordinates $(r, \theta)$ on the disk, we have
\begin{equation} \label{eqn:bdry}
      u_{zz}^\perp = \frac{1}{4} e^{-2i\theta} \left[ u_{rr}^\perp -\frac{1}{r^2} u_{\theta \theta}^\perp
                                                                                   - \frac{2i}{r} u_{r\theta}^\perp \right].
\end{equation}                                                                                  
By the free boundary condition, $u_r$ is orthogonal to $\p B$ along the boundary of the disk, and so $u_r =\lambda u$ for some function $\lambda$ on $\p D$. Therefore on $\p D$,
\[
     u_{r\theta}= \lambda_{\theta} u + \lambda u_\theta = \frac{\lambda_{\theta}}{\lambda} u_r + \lambda u_{\theta},
\]
and so $u_{r \theta}^\perp =0$ on $\p D$. It then follows from (\ref{eqn:bdry}) that $z^4 (u_{zz}^\perp)^2$ is real on $\p D$. Since $z^4 (u_{zz}^\perp)^2$ is holomorphic on $D$, $z^4 (u_{zz}^\perp)^2$ must be constant. But  $z^4 (u_{zz}^\perp)^2$ vanishes at the origin, and so $z^4 (u_{zz}^\perp)^2 \equiv 0$.

In particular, on $\p D$ we have
\[
      0=z^4 (u_{zz}^\perp)^2=\frac{1}{16} (u_{rr}^\perp -u_{\theta \theta}^\perp)^2.
\]      
By the minimality, this implies that $u_{rr}^\perp = u_{\theta \theta}^\perp=0$ on $\p D$. Therefore, the second fundamental form of $\Sigma$ is zero on $\p \Sigma$. It follows that the second fundamental form of $\p \Sigma$ in $S^{n-1}$ is zero. Therefore $\p \Sigma$ is a great circle, and $\Sigma$ must be an equatorial plane disk.
\end{proof}

More generally

\begin{theorem}
Let $u: D \rightarrow B$ be a proper branched immersion with parallel mean curvature vector, from the disk $D$ to a ball $B$ in an $n$-dimensional  space $N$ of constant curvature, such that $u(D)$ meets $\p B$ orthogonally. Then $\Sigma$ is contained in a 3-dimensional totally
geodesic submanifold, and is totally umbilic.
\end{theorem}

\begin{proof}
Let $\n$ be the pull back connection in the pull back bundle $u^*(TN)$, let $\n^\perp$ be the connection in the pull back $u^*(N\Sigma)$ of the normal bundle to $\Sigma=u(D)$ in $N$, and let $A$ be the second fundamental form of $\Sigma$ in $N$. 
Let $z=x_1+ix_2$ be any local complex coordinate on $\Sigma$, and 
consider
\begin{equation} \label{Azz1}
      A(\der{z},\der{z})=\n^\perp_{\der{z}} du(\der{z})
      = \frac{1}{4} \left[ (A_{11}-A_{22} ) -2iA_{12} \right]
 \end{equation}
where $A_{ij}=A(\der{x_i},\der{x_j})$, $i, \, j=1,\, 2$, are the coefficients of the second fundamental form in the coordinate basis. We let $A_{ij;k}$ denote the covariant derivative $(\nabla^\perp_{\der{x^k}} A) (\der{x^i},\der{x^j})$. Given $p \in \Sigma$, we may choose local complex coordinates near $p$ such that $\lambda(p)=1$ and $\nabla \lambda (p)=0$, where $\lambda=\left|\fder{u}{x^i}\right|^2$ is the conformal factor. Then, computing at $p$,
\begin{align*}
        \n^\perp_{\der{\bar{z}}} A(\der{z},\der{z}) 
        &= \frac{1}{8} \left[ A_{11;1} -A_{22;1}+2A_{12;2} +i(A_{11;2}-A_{22;2}-2A_{12;1})\right]  \\
        &=\frac{1}{8}\left[ -2A_{22;1}+2A_{12;2} + i ( 2A_{11;2}-2A_{12;1}) \right] \\
        &=0
\end{align*}
where the second equality follows since at $p$, $A_{11;i}+A_{22;i}=0$, $i=1,\,2$ by the assumption that $\Sigma$ has parallel mean curvature (at a general point $(\lambda^{-1}A_{11}+\lambda^{-1}A_{22})_{;i}=0$), and the last equality follows from the Codazzi equations since the ambient manifold has constant sectional curvature. 

If we had chosen another complex coordinate $\zeta$ we would have by the chain rule 
\[  A(\der{\zeta},\der{\zeta})= (\fder{z}{\zeta})^2A(\der{z},\der{z}),
\]
and since the right hand side is holomorphic at $p$ in the $z$ coordinates it follows that it
is also holomorphic with respect to $\zeta$ at $p$. Since $p$ was an arbitrary point we have
shown that for any complex coordinate $z$ at any point we have
 \[ \n^\perp_{\der{\bar{z}}} A(\der{z},\der{z})=0.
 \] 
 
It now follows from metric compatibility of the connection that the function
\[
       \phi(z):=A(\der{z},\der{z})\cdot A(\der{z},\der{z})
\]
is holomorphic on $D$.

Consider polar coordinates $(r,\theta)$ on the disk, so $z=re^{i\theta}$. Set $w=\log z=\log r + i \theta$. Then 
\[ \der{z}=\fder{w}{z}\der{w}=\frac{1}{z}\der{w},\] and 
\[ A(\der{z},\der{z})=\frac{1}{z^2}A(\der{w},\der{w}).\]
On $\p D$ we have $\der{w}=\der{r}+i\der{\theta}$, and
\begin{equation} \label{Azz2}
          z^2A(\der{z},\der{z})=A(\der{w},\der{w}) 
          =A(\der{r},\der{r})-A(\der{\theta},\der{\theta})-2iA(\der{r},\der{\theta}).          
\end{equation}
By the free boundary condition, the outward unit normal $\eta$ along $\p \Sigma$ is normal to the sphere. Thus, since $du(\der{r})=\lambda \eta$, where $\lambda=|du(\der{r})|$, we have
\[
        A(\der{r},\der{\theta})=\n^\perp_{\der{\theta}}du(\der{r})=[-\lambda S(du(\der{\theta}))]^\perp,
 \]   
where $S$ is the shape operator of $\p B$. But the shape operator of a sphere in a space of constant curvature is a multiple of the identity. Therefore, 
\begin{equation} \label{boundary}
         A(\der{r},\der{\theta})=0 \mbox{ on } \p D,
\end{equation}
and so $z^2A(\der{z},\der{z})$ is real on $\p D$.
 
Thus, $z^4 \phi(z)$ is a holomorphic function on $D$ that is real on $\p D$. This implies that $z^4 \phi(z)$ is constant. But $z^4\phi(z)$ vanishes at the origin, and so $z^4\phi(z)\equiv 0$. In particular, by (\ref{Azz2}) and (\ref{boundary}) we see that on $\p D$, 
\[
      0=z^4\phi(z)=\left[A(\der{r},\der{r})-A(\der{\theta},\der{\theta})\right]^2.
\]
Therefore, $z^2A(\der{z},\der{z})=0$ on $\p D$. But $z^2A(\der{z},\der{z})$ is a holomorphic section of the normal bundle, and so $z^2A(\der{z},\der{z})\equiv 0$ on $D$. To see this, choose a global frame $s_1, \ldots, s_{n-2}$ for the normal bundle $u^*(N\Sigma)$ over the disk $D$. Then,
\[
         z^2A(\der{z},\der{z})=\sum_{j=1}^{n-2} f_js_j
\]
for some complex-valued functions $f_j$ on $D$, $j=1, \ldots, n-2$, and
\[
        0 = \n^\perp_{\der{\bar{z}}} \left(z^2A(\der{z},\der{z})\right) 
        = \n^\perp_{\der{\bar{z}}}\left(\sum_{j=1}^{n-2} f_js_j\right) 
        = \sum_{j=1}^{n-2} \left( \fder{f_j}{\bar{z}} + \sum_{k=1}^{n-2} a_{kj} f_k \right) s_j,
\]
where $a_{jk}$, $j,\,k=1, \ldots n-2$, are given by 
$\n^\perp_{\der{\bar{z}}} s_j = \sum_{k=1}^{n-2} a_{jk} s_k$. Therefore, we have
\[ 
   \left\{ \begin{array}{llc}
      \frac{\p f_j}{\p \z}+ \sum_{k=1}^{n-2} a_{kj} f_k =0 
             & \mbox{on}\;D &  \\
             & & \text{for}\; j= 1, \ldots, n-2. \\
             f_j =0 & \mbox{on} \; \p D & 
               \end{array}
       \right. 
\]
This implies that $f_j \equiv 0$ on $D$ for $j=1,\ldots , n-2$, and so $z^2A(\der{z},\der{z})\equiv 0$ on $D$. 

Hence, $A(\der{z},\der{z})\equiv 0$ on $D$.
By (\ref{Azz1}), this implies that $A_{11}=A_{22}$ and $A_{12}=0$ on $D$, or $A_{ij}=\frac{1}{2}H \delta_{ij}$. Consider the three dimensional linear space $\mathcal{V}(x)$ at each point $x$ of $D$ spanned by the tangent space $T_{u(x)}\Sigma$ and the mean curvature vector $H_{u(x)}$. We claim that $\mathcal{V}$ is parallel. To see this, let $X \in \Gamma(u^*(TN))$ be a vector field along $\Sigma$ such that $X_x \in \mathcal{V}(x)$ for all $x \in D$. Then $X=X^T+fH$, where $X^T$ is tangent to $\Sigma$ and $f$ is a smooth function on $D$. Then, if $Y \in \Gamma(u^*(T\Sigma))$,
\begin{align*}
       \n_Y X &= (\n_Y X^T)^T + \n^\perp_Y X^T + X(f) H + f (\n_Y H)^T +f \n^\perp H \\
       &=(\n_Y X)^T + \sum_{i,\,j=1}^2 c_{ij} A_{ij} + X(f) H + f (\n_Y H)^T   
       \qquad \mbox{ (since $H$ is parallel)} \\
       &=(\n_Y X)^T + c H + X(f) H + f (\n_Y H)^T   
       \qquad \mbox{ ( since $A_{ij}=\frac{1}{2}H \delta_{ij}$)} \\
        & \in \mathcal{V}.
\end{align*}  
Therefore, $\mathcal{V}$ is parallel. It follows that $\Sigma$ is contained in a 3-dimensional 
submanifold $N_0$ of $N$. To see this, let $p=u(0)$, and let $N_0$ be the totally geodesic
three dimensional submanifold containing $p$ and tangent to $\mathcal{V}(p)$. If
we fix a basis $e_1,e_2,e_3$ for $\mathcal{V}(p)$, then the
submanifold $N_0$ can be characterized as the set of points $q$ of $N$ which can be
joined by a curve $\gamma(t)$ from $p$ to $q$ such that $\gamma'(t)$ is a linear
combination of the parallel transports $e_1(t),e_2(t),e_3(t)$ for each $t$. Since $\mathcal{V}$
is parallel along $\Sigma$ and contains the tangent space to $\Sigma$ at each point,
this property holds for all $q\in\Sigma$. Therefore $\Sigma\subset N_0$.

The classification of totally umbilic surfaces in three-dimensional space forms is well known (see \cite{Sp}). In the case of $\mathbb{R}^n$, $\Sigma$ must be an equatorial plane or a spherical cap meeting $\p B$ orthogonally; in particular, if $\Sigma$ is minimal, $\Sigma$ must be an equatorial plane disk in the ball. In the case of $\mathbb S^n$, $\Sigma$ must be a spherical cap meeting $\p B$ orthogonally; in particular, if $\Sigma$ is minimal, it must be a totally geodesic disk in $B$. In the case of hyperbolic space, $\Sigma$ must be a totally geodesic plane disk or its equidistant, a horosphere, or a round sphere meeting $\p B$ orthogonally; in particular, if $\Sigma$ is minimal, then $\Sigma$ must be a totally geodesic plane disk in $B$.
\end{proof}
  
\begin{corollary}
Let $\Sigma$ be an immersed minimal $\mathbb S^2$ in $\mathbb S^n$ that is invariant under reflection through a totally geodesic $\mathbb S^{n-1}$. Then $\Sigma$ is totally geodesic.
\end{corollary}
\begin{proof} Since the surface $\Sigma$ intersects the equatorial $\mathbb S^{n-1}$
orthogonally, it is transversal. Thus the intersection consists of smooth circles. If we
consider an inner-most circle (one which bounds a disk disjoint from the other circles),
then we get a free boundary minimal disk in a hemisphere, and by Theorem 2.2 this
disk is totally geodesic. It follows that $\Sigma$ is the corresponding totally geodesic
minimal $\mathbb S^2$ since it is assumed to be reflection symmetric.
\end{proof}
     
\bibliographystyle{plain}

\end{document}